\documentclass[11pt]{amsart}

\usepackage{tikz}
\usepackage[
mode=image
]{standalone}

\usepackage{amsfonts}
\usepackage{mathrsfs}
\usepackage{amsthm}
\usepackage{graphicx}
\usepackage{placeins}
\usepackage{mathtools}
\usepackage{hyperref}
\usepackage{setspace}

\usepackage[utf8]{inputenc} \usepackage[
sortcites=true,
backend=biber]{biblatex}
\renewbibmacro{in:}{}

\usepackage{xcolor}

\makeatletter
\newtheorem*{rep@theorem}{\rep@title}
\newcommand{\newreptheorem}[2]{%
\newenvironment{rep#1}[1]{%
 \def\rep@title{#2 \ref{##1}}%
 \begin{rep@theorem}}%
 {\end{rep@theorem}}}
\makeatother

\newtheorem{theorem}{Theorem}
\newreptheorem{theorem}{Theorem}
\newtheorem{proposition}[theorem]{Proposition}
\newtheorem{lemma}[theorem]{Lemma}
\newtheorem{corollary}[theorem]{Corollary}
\theoremstyle{definition}

\usepackage{subcaption}

\usepackage{setspace} %

\newcommand{\Dknot}{\prescript{1}{}{\mathcal{D}}}

\newcommand{\AKT}{\prescript{1}{}{\!\mathcal{A}}}
\newcommand{\akt}{\prescript{1}{}{\mathcal{T}}}

\newcommand{\AKTnum}{(\prescript{1}{}{\!A})}
\newcommand{\aktnum}{(\prescript{1}{}{a})}

\newcommand{\Xtang}{\mathfrak{X}}

\DeclareFieldFormat{url}{%
  \iffieldundef{doi}{%
    \mkbibacro{URL}\addcolon\space\url{#1}%
  }{%
  }%
}

\DeclareFieldFormat{urldate}{%
  \iffieldundef{doi}{%
    \mkbibparens{\bibstring{urlseen}\space#1}%
  }{%
  }%
}

\begin{document}
\title[]{On the structure and scarcity of alternating knots}
\author{Harrison Chapman} \email{hchaps@gmail.com}
\address{Department of Mathematics\\
 Colorado State University, Fort Collins CO}
\date{\today}

\begin{abstract}  
  Given a class of objects, a pattern theorem is a powerful result describing their structure. We show that alternating knots exhibit a pattern theorem, and use this result to prove a long-standing conjecture that alternating knots grow rare. This is currently the best possible analogue of a pair of theorems on alternating links of Sundberg and Thistlethwaite in 1998 and Thistlethwaite in 1998, given the current obstructions to an exact enumeration of knot diagrams. We also discuss implications of this pattern theorem for subknots and slipknots in minimal alternating knot diagrams and types, partially answering a conjecture of Millett and Jablan.
\end{abstract}

\maketitle

Following Menasco and Thistlethwaite's proof of the Tait flyping conjecture for alternating links~\cite{Menasco_1991,Menasco_1993}, Sundberg and Thistlethwaite proved an impressive formula~\cite{Sundberg_1998} (see also~\cite{Zinn_Justin_2002} for an illuminating alternate proof using techniques from random matrix theory) for the exponential growth rate of the number \(A_n\) of alternating prime link types, that:
\begin{equation}
  \lim_{n\to\infty}{A_n^{1/n}} = \frac{101+\sqrt{21001}}{40}.
  \label{eq:altlinktypecount}
\end{equation}
Soon after, Thistlethwaite leveraged similar techniques to prove that alternating prime link types are asymptotically exponentially rare among all prime link types~\cite{Thistlethwaite_1998}.

In 1998 Hoste, Thistlethwaite, and Weeks published the results of their impressive tabulation of all 1,701,936 knot types up to 16 crossings~\cite{Hoste98}. Their data provide \emph{exact} ratios of alternating knot types to general knot types; \emph{all} knot types are alternating until 8 crossings, at which case the ratios decrease so that 27\% of 16-crossing knot types are alternating. From this they mention that, ``it is plausible that the proportion of knots which are alternating tends exponentially to zero with increasing crossing number.'' Two decades later, this claim has yet to see proof. Our goal herein is to concretely show:
\begin{theorem}
  All but exponentially few prime knot types are nonalternating.
  \label{thm:altknotrare}
\end{theorem}

This case of knots is far more difficult: The proof of Equation~\ref{eq:altlinktypecount} relies heavily on Tutte's~\cite{Tutte1963} exact enumeration of 4-valent planar maps. There is as of yet no similar result for alternating prime knot types. Any likely proof strategy would require an enumeration of the subclass of 4-valent planar maps called plane curves, which does not yet exist beyond conjecture~\cite{Schaeffer2004} and experiment~\cite{Jacobsen2002,ZinnJustin2009,Chapman2016:mc}. Note that if the exponential growth rate of the number of knot types exists, it lies within the (best-known) bounds of 2~\cite{Ernst_1987} and 10.39~\cite{Stoimenow_2004}.

In the absence of such a result, particularly in the study of knotted objects, most knowledge is based on ``pattern theorems''~\cite{Kesten1963,Kesten1964} which describe the structure of almost all large objects. Such pattern theorems are in some sense a relaxation of a precise enumeration. Notably, pattern theorems lie at the heart of many proofs of the asymptotic certainty of knotting~\cite{Sumners_1988,Pippenger89,Diao_1994,Chapman2016} in a wide variety of models of random knots. For an overview of random models of knots, see for instance~\cite{Orlandini07,Even_Zohar_2017}. In proving Theorem~\ref{thm:altknotrare} we hope additionally to convince the reader that pattern theorems are a powerful tool capable of answering topological and geometric questions with an underlying combinatorial structure.

The following pattern theorem has been known about reduced prime alternating knot diagrams through their bijection with prime plane curves~\cite{Chapman2016}:
\begin{theorem}
  Let \(P\) be a prime tangle which may be found in a reduced prime alternating knot diagram. Then there exists a constant \(c > 0\) so that all but exponentially few reduced prime alternating knot diagrams with \(n\) crossings contain \(cn\) instances of \(P\).
  \label{thm:akdpatterntheorem}
\end{theorem}
By accounting for flypes, we will prove the stronger result for alternating prime knot \emph{types}:
\begin{theorem}
  Let \(P\) be a prime 3-edge-connected alternating tangle diagram. Then there exists a constant \(c > 0\) so that all but exponentially few prime alternating knot types \(K\) have a minimal alternating diagram which contains at least \(cn\) instances of \(P\). Furthermore, \emph{every} minimal diagram for such an alternating knot type contains at least \(cn\) copies of \(P\) or its reflections.
  \label{thm:aktpatterntheorem}
\end{theorem}
The distinction is subtle: Theorem~\ref{thm:akdpatterntheorem} allows for the unlikely possibility that only diagrams for a relatively small number of alternating knot types obey a pattern theorem. Instead, we prove the latter, stronger theorem using a new pattern theorem construction.

After showing that alternating knot types are rare, we conclude with an application of the pattern theorems to a pair of conjectures of Millett and Jablan~\cite{Millett_2016,Millett_2017}: That almost every prime knot diagram contains a trefoil as a subknot (Conjecture 1) and as a slipknot (Conjecture 2). Our pattern theorems provide answers in the case of alternating knots.

\section{Definitions}
\label{sec:definitions}

A \emph{link diagram} is a 4-valent spherically embedded multigraph (planar map), together with over- undercrossing data at each vertex, called a \emph{crossing}. A \emph{link component} of a diagram is an equivalence class of edges, modulo meeting on opposite sides of a crossing. A link diagram is a \emph{knot diagram} if it has precisely one link component. A diagram is \emph{reduced} if it has no disconnecting vertices.

A \emph{knot type} is an equivalence class of diagrams under local Reidemeister moves, and Reidemeister's theorem~\cite{Alexander_1926} equates this with the usual theory of time knots in space. A knot diagram is \emph{alternating} if, when following the knot component around some orientation, crossings are encountered in the periodic over-under sequence. A knot type is alternating if some representative knot diagram is alternating. Knot types admit prime decompositions under the binary connected sum relation, and a knot type is \emph{prime} if it has only one component under this decomposition.

Let \(\AKT = \bigcup_{n} {\AKT_n}\) denote the class of alternating prime knot types, grouped by minimal crossing number, and let \(\AKTnum_n = |\AKT_n|\) be their counting sequence.

A \emph{tangle diagram} is a 4-valent multigraph embedded in the disk \(D^2\) (planar map with boundary) with four extremal vertices on \(\partial B\), together with over- and undercrossing information at each interior vertex. A \emph{tangle type} is an equivalence class of tangle diagrams under the Reidemeister moves in the disk interior. Tangle diagrams are \emph{alternating} if every sequence of signs along a component alternates between threading over and under. Alternating reduced tangle diagrams, like alternating reduced link diagrams, satisfy the Tait flyping conjecture~\cite{Menasco_1991,Menasco_1993,Sundberg_1998}; two alternating reduced tangle diagrams represent the same tangle type if and only if they are related by a series of flype moves that leave \(\partial B\) fixed. Furthermore, this implies that an alternating diagram is minimal if and only if it is reduced.

Except otherwise mentioned, all diagrams discussed in this paper will be alternating. This permits us to draw figures of flat tangle ``shadows'', where the crossing signs are then determined by the external structure. In diagrams representing tangles, the convention will be that the first crossing that the lower-left leg passes will be \emph{over}. In the case of substructure, it will be assumed that the crossing signs are consistent with the alternating pattern.

\begin{figure}
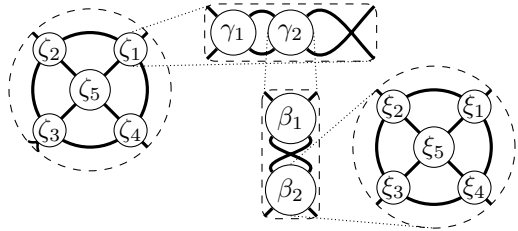

  \centering
  \includestandalone{figs/tangle_decomposition}
  \caption{Tangles can be decomposed into finite, repeated compositions of tangles of types I and II.}
  \label{fig:tangledecomp}
\end{figure}

We briefly describe the decomposition of tangles of~\cite{Sundberg_1998} which is summarized in Figure~\ref{fig:tangledecomp}. Two tangles admit horizontal and vertical \emph{tangle sum}, described in Figure~\ref{fig:tanglesum}. Broadly, a tangle is \emph{type II} if it is either a crossing or a tangle sum of two nontrivial tangles; otherwise it is \emph{type I}. If a tangle is type I, then one can identify the interior dual 4-cycles which are maximal in the sense that, if all such dual 4-cycle interiors are collapsed simultaneously to vertices, one obtains a disk embedded multigraph with no faces of degree 1 or 2 (called a \(c\)-net or a Conway polyhedron). The interiors of each such dual 4-cycle is itself a tangle. On the other hand, a type II tangle is either simply a crossing, or it can be viewed as either a horizontal or vertical tangle sum of tangles. In the horizontal sum decomposition, any crossings will be called \emph{flypable crossings}.

\begin{figure}
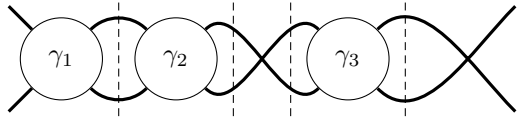

  \centering
  \includestandalone{figs/tangle_sum}
  \caption{A horizontal tangle sum of nontrivial tangles \(\gamma_1, \gamma_2, \gamma_3\) and two flypable crossings.}
  \label{fig:tanglesum}
\end{figure}

Tangle diagrams and types have the counterclockwise ordering of their endpoints fixed. If we root the tangle at the lower-left hand corner, there are three classes of tangle distinguished by where the strand rooted at the lower-left ends. If its other endpoint is at;
\begin{enumerate}
\item the lower-right hand corner, the tangle is be called \emph{horizontal},
\item the upper-left hand corner, the tangle is called \emph{vertical}, or
\item the upper-right hand corner, the tangle is called \emph{crossing}.
\end{enumerate}
Note that the numbers of vertical and horizontal tangles are the same by considering a \(90^\circ\) rotation and a re-rooting. These three subclasses are depicted in Figure~\ref{fig:altknot_tangleclasses}.

\begin{figure}
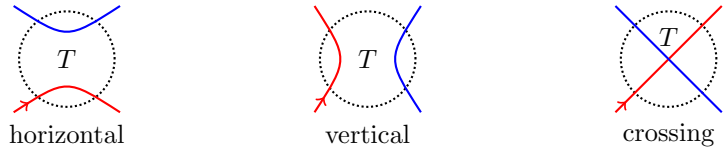

  \centering
  \includestandalone{figs/altknot_tangleclasses}
  \caption{The three subclasses of tangle by exterior strand arrangment.}
  \label{fig:altknot_tangleclasses}
\end{figure}

Tangle diagrams may be closed into links through the one-crossing closure, see Figure~\ref{fig:altknot_closure}. This also defines a projection from alternating tangle types to alternating link types. A tangle (diagram, type) which closes to a knot (diagram, type) (i.e. has precisely one link component after closure) will be called a \emph{knot tangle (diagram, type)}. The class of knot tangle diagrams is exactly the union of the sets of horizontal and vertical tangles which have \emph{no interior link components}. On the other hand, notice that under the one-crossing closure, a crossing-class tangle necessarily has at least two closed link components and cannot be a knot. Let \(\akt = \bigcup_{n} {\akt_n}\) be the number of alternating prime knot tangle types counted by minimal crossing number, and let \(\aktnum_n = |\akt_n|\) be their counting sequence.

\begin{figure}
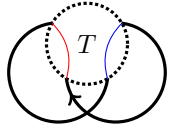

  \centering
  \includestandalone{figs/altknot_closure}
  \caption{The one-crossing closure of an alternating tangle.}
  \label{fig:altknot_closure}
\end{figure}

As alternating prime knot types are a subclass of alternating prime link types and alternating prime knot tangle types are a subclass of alternating prime tangle types, Corollary 3.3.1 of Sundberg and Thistlethwaite~\cite{Sundberg_1998} provides the relation:

\begin{proposition} One has that,
  \[ \frac{\aktnum_{n-1}}{8(2n-3)} \le \AKTnum_n \le \frac{\aktnum_{n-1}}{2}. \]
  \label{thm:aktknottanglinear}
\end{proposition}

We will hence focus our attention to alternating knot tangles, from which alternating knot results will follow immediately.

\section{Results}
\label{sec:results}

\subsection{Patterns in alternating knot types}
\label{sec:patternthm}

For a strong pattern theorem result on alternating knot tangle types, existence of their exponential growth rate must be known to exist (for the knot diagram setting, see~\cite{Chapman2016}).

\begin{theorem}
  \label{thm:aktgrowth}
  Let \(\aktnum_n\) denote the number of alternating tangle types whose one-crossing closure produces an alternating knot type. Then \(\lim_{n\to\infty}{\aktnum_n^{1/n}}\) exists.
\end{theorem}

\begin{proof}
  Two alternating knot tangles \(T_1,T_2\) can be composed to produce a new, unique alternating knot tangle: If they are both horizontal, then composition as in in Figure~\ref{fig:altknot_supermult_horiz} produces a new horizontal knot tangle. If one is vertical and the other horizontal, the same composition produces a vertical knot tangle. In the case where both tangles are vertical the composition would produce an extra link component; instead they must be inserted into a rotated structure as in Figure~\ref{fig:altknot_supermult_vert}. In all cases the resulting tangle is uniquely decomposable; \(T_1, T_2\) can be recovered from the new diagram as separated by the unique dual 2-cycle as no flypes can alter the external structure of the tangle.

  \begin{figure}
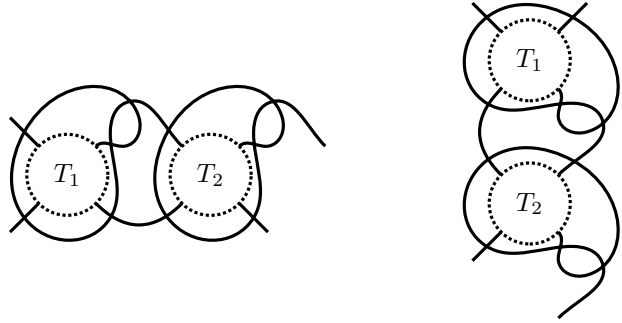

    \centering
    \begin{subfigure}[t]{0.3\textwidth}
      \centering
      \raisebox{.45in}{
        \includestandalone{figs/altknot_supermult}}
      \subcaption{Composition for HH, HV, and VH pairs of tangles}
      \label{fig:altknot_supermult_horiz}
    \end{subfigure}\hfil
    \begin{subfigure}[t]{0.3\textwidth}
      \centering
      \includestandalone{figs/altknot_supermult_vert}
      \subcaption{Composition for VV pairs of tangles}
      \label{fig:altknot_supermult_vert}
    \end{subfigure}
    \caption{A composition for alternating knot tangles which adds precisely 10 crossings.}
    \label{fig:altknot_supermult}
  \end{figure}

  This composition provides a supermultiplicativity type relation \[\aktnum_{n+m+10} \ge \aktnum_{n}\aktnum_{m}.\] Together with the upper bound by the numbers \(a_n\) of \emph{all} alternating tangles (for large enough \(n\)) from~\cite{Sundberg_1998} \[\aktnum_{n}^{1/n} \le a_n^{1/n} \sim \frac{101+\sqrt{21001}}{40}, \] a strengthening of Fekete's lemma~\cite{Wilker1979} then applies with \(f(m)=m+10\).
\end{proof}


By the above theorem and Proposition~\ref{thm:aktknottanglinear} we immediately have,

\begin{corollary}
  \label{cor:AKTnumexists}
  The exponential growth rates of alternating knot types and alternating knot tangle types are the same:
  \[ \lim_{n\to\infty}\AKTnum_n^{1/n} = \lim_{n\to\infty}\aktnum_n^{1/n}. \]
\end{corollary}

With our result on exponential growth rates in hand, we now state the pattern theorem for alternating knot tangle types and alternating knot types. A tangle or knot type \([T]\) \emph{contains} a tangle diagram \(R\) if there exists some minimal diagram for \(T\) containing \(R\) as a sub-diagram.

\begin{reptheorem}{thm:aktpatterntheorem}
  Let \(R\) be a reduced type-I alternating horizontal knot tangle diagram which admits no interior flypes. Then there exist constants \(c > 0\), \(1 > d > 0\), and \(N \in \mathbb{N}\) such that for all \(n \ge N\), with \(h_n\) the number of alternating knot tangle \emph{types} with no minimal diagram having \(\ge cn\) copies of \(R\), we have the relation,
  \[ \frac{h_n}{\aktnum_n} < d^n. \]
  Furthermore, if \(H_n\) is the number of alternating knot \emph{types} with no minimal diagram having \(\ge cn\) copies of \(R\), then
  \[ \frac{H_n}{\AKTnum_n} < d^n. \]
  Namely, the fraction of alternating knot types whose minimal diagrams contain fewer than \(cn\) copies of \(R\) or its reflections is exponentially small.
\end{reptheorem}

We first equate the classes of alternating knot tangle types with a specific class of alternating knot tangle diagrams. The proof then becomes entirely diagrammatic, and relies only on an attachment construction of any such tangle \(R\).

Using flype equivalence of minimal prime alternating diagrams~\cite{Menasco_1991,Menasco_1993}, fix once and for all a diagram representation \(T\) as follows:

\begin{figure}
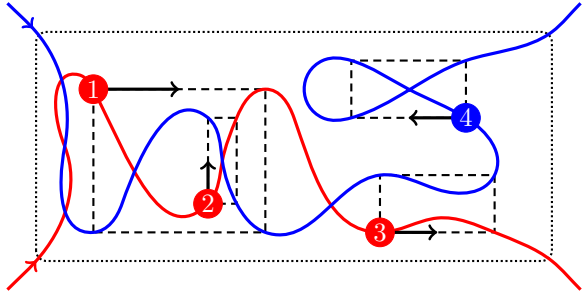

  \centering
  \includestandalone{figs/altknot_stdori}
  \caption{Following the string with a well-defined orientation, all horizontal tangle sums can be oriented left-to-right in a well-defined manner. Dashed boxes correspond to maximal horizontal tangle sums. The paths of the strings is a rough depiction, and is not meant to imply any diagram structure beyond entries into and out of tangle sum levels.}
  \label{fig:altknot_stdori}
\end{figure}

Starting with the lower-left exterior leg of the tangle \(T\), follow the string component around the tangle diagram one edge at a time. Upon first entry of a maximal non-trivial horizontal tangle sum, fix the left-to-right orientation of the tangle sum such that the edge being traversed is on the left side of the tangle sum (a horizontal direction is already determined by the sum decomposition). Continue traversing the tangle, until reaching the lower-right (resp. upper-left) exterior leg if \(T\) is horizontal (resp. vertical). If some maximal horizontal tangle sums are still unoriented, take the upper-left (resp. lower-right) leg, and follow the strand edge-by-edge applying the same orientation. See Figure~\ref{fig:altknot_stdori} for an example of this orientation procedure.

Using the orientations of all maximal tangle sums in \(T\), apply flypes to each tangle sum region so that all flypable crossings are on the right hand side, as in Figure~\ref{fig:altknot_stdform}. The resulting diagram is in \emph{standard form}.

\begin{lemma}
  Every reduced alternating knot tangle type admits precisely one standard form diagram. This provides a bijection between the class of all alternating reduced knot tangle types and the class of standard form alternating reduced knot tangle diagrams.
\end{lemma}

\begin{proof}
  Let \(T_1, T_2\) be two standard form alternating reduced knot tangle diagrams for the alternating reduced knot tangle type \([T]\). All alternating reduced tangle diagrams decompose by the discussion of Section~\ref{sec:definitions} as a finite composition of alternating reduced tangle diagrams, with the property that any flype equivalences are confined to the specific composition level containing the flyping crossing.

  The exterior composition levels of \(T_1, T_2\) are either both tangle sums or not. If they are not both tangle sums, then they must be the same as flyping can only happen in tangle sums. If they are, then an orientation of the sum has been decided by the scheme above. As \(T_1\) and \(T_2\), all flypable crossings are on the right-hand side of the sum. Furthermore, as the tangle summands cannot be permuted, they are in the same order on the left-hand side. Hence, in this case the exterior composition levels of both diagrams are the same.

  A finite recursion using the above argument yields that both \(T_1\) and \(T_2\) are identical diagrams. Hence \([T]\) has a unique standard form.
\end{proof}

\begin{proof}[Proof of Theorem~\ref{thm:aktpatterntheorem}]
  
  By the above lemma, we may view \(\akt\) as a class of diagrams. To prove a pattern theorem for \(R\) in \(\akt\), it now remains to prove that there is a viable attachment scheme of \(R\) into diagrams in standard form that preserves the standard form.

  The definition of \(\Xtang\) depends on whether \(R\) is a crossing tangle (see Figure~\ref{fig:altknot_conwayshield_cross}) or not (see Figure~\ref{fig:altknot_conwayshield_vert}). The attachment scheme is summarized in Figure~\ref{fig:altknot_tangann}. Any rooted dual 4-cycle in \(T\) can be replaced by the annulus \(\Xtang\) so that the lower-left corner of \(\Xtang\) oriented inward and the root of the dual cycle agree. Notice that the structure of \(\Xtang\) has that each strand going in a given corner of the outer annulus boundary comes out \emph{the same} respective corner on the interior.

  \begin{figure}
    \centering
    \includestandalone{figs/altknot_shieldconway_cross}
    \caption{Definition of \(\Xtang\) for crossing tangles \(R\).}
    \label{fig:altknot_conwayshield_cross}
  \end{figure}

  \begin{figure}
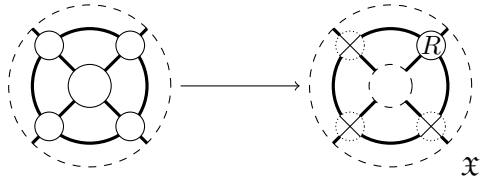

    \centering
    \includestandalone{figs/altknot_shieldconway_vert}
    \caption{Definition of \(\Xtang\) for horizontal and vertical tangles \(R\).}
    \label{fig:altknot_conwayshield_vert}
  \end{figure}

  We now show that, at minimum, there are \(n\) attachment locations for any given standard form tangle diagram which (a) produce a new standard form tangle diagram and (b) may be performed in parallel. An attachment location is chosen by, within a maximal horizontal tangle sum domain with \(k \ge 1\) flypable crossings in standard form, picking \(1 \le \ell \le k\) and identifying the dual 4-cycle that passes immediately to the left of the leftmost (first) flypable crossing and also to the right of the \(\ell\)th flypable crossing. Additionally, for any crossing which is not contained in a horizontal tangle sum, a single rooted dual 4-cycle may be chosen for insertion about the crossing. The root of the dual cycle is taken to be the lower-left corner, with respect to the horizontal orientation given by the standard form. Such an insertion can be performed in parallel; after choosing \(\{\ell_i\}_{i=1}^k\), take the \(\{\ell_i\}_{i=1}^k\) as sorted and identify the unique \(k\) concentric dual cyles, all of which pass through to the left of the first flypable crossing and so that the \(i\)th passes to the right of the \(\ell_i\)th flypable crossing. Notice that the number of attachment sites which we have identified is,
  \[
    \sum_{\textrm{all max tangle sums \(\gamma\)}}{\textrm{flypable crossings in \(\gamma\)}} = \#{\textrm{ crossings in diagram}} = n,
  \]
  as every flypable crossing is contained in precisely one maximal tangle sum.
  
  \begin{figure}
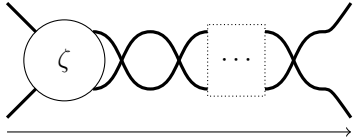

    \centering
    \includestandalone{figs/altknot_stdform}
    \caption{Any sub-portion of a tangle counted by \(\gamma\), i.e.\ one which is a horizontal sum of tangles, can be realized in standard form where all flypable crossings are on the right, and a tangle \(\zeta\) consisting of no flypable crossing summands on the left.}
    \label{fig:altknot_stdform}
  \end{figure}




  \begin{figure}
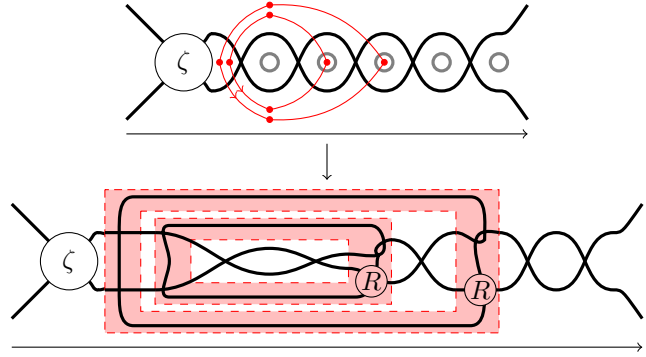

    \centering
    \includestandalone{figs/altknot_tangcycle}
    \caption{Insertion of a crossing tangle \(R\) at marked rooted dual-cycles. By our scheme, there are as many choices for dual cycles as crossings, determined entirely choice of face. Faces which can be chosen are marked with gray hollow circles. by Insertion of a horizontal or vertical tangle is the same, with a different inserted annulus \(\Xtang\).}
    \label{fig:altknot_tangann}
  \end{figure}
  
  After identifying any number of such rooted dual cycles, the tangle \(R\) is inserted by replacing the oriented dual cycle with the annulus \(\Xtang\), as in Figure~\ref{fig:altknot_tangann}. What is most important for this proof is that if the original tangle diagram is in standard form, \emph{so too is the resultant diagram}, no matter how many copies of \(R\) are attached in parallel---this follows from (a) that \(R\) and so \(\Xtang\) admit no flypes, (b) the annulus \(\Xtang\) does not apply any nontrivial permutation to the exterior and interior termini of its strands, and (c) there are no new horizontal flypes to perform as all dual cycles are chosen with regards to the orientation of the standard form diagram.

  There are then four properties we must check for this attachment scheme:
  \begin{enumerate}
  \item \emph{At least \(n\) non-conflicting places of attachment exist.} There are precisely \(n\) attachment locations by the scheme described as each crossing appears in exactly one maximal horizontal tangle sum. As we have described how to perform multiple simultaneous attachments, they are non-conflicting.
    
  \item \emph{Only diagrams in \(\akt\) are produced.} Such an attachment can only produce new alternating knot diagrams, and the resulting diagrams are in standard form by the discussion above.

  \item \emph{For any diagram produced as such the copies of \(R\) may be identified and they are vertex disjoint.} Vertex disjointness is guaranteed by the primality of \(R\) and the annular shield \(\Xtang\). Identification of an attached copy of \(R\) in a diagram in standard form involves precisely identifying the copy of \(\Xtang\) that was added and its lower-left corner, which can be determined by the structure of \(\Xtang\).

  \item \emph{Given any number of copies of \(R\) which have been inserted, the original diagram and associated places of attachment are uniquely determined.} Given an annular region \(\Xtang\), collapsing the annulus into a rooted dual cycle both removes the copy of \(R\) (producing a standard form diagram to which such structure could have been added) as well as the rooted dual cycle where \(R\) was attached.
  \end{enumerate}

  This proves that the attachment of \(R\) satisfies the appropriate hypotheses. Hence the pattern theorem for \(R\) in alternating knot tangle types follows by the general diagram pattern theorem; Theorem 6 of~\cite{Chapman2016}. By the linear bounds on the relationship between the numbers of alternating tangle types and alternating knot types of Proposition~\ref{thm:aktknottanglinear}, the latter equation for alternating knot types holds as well.
\end{proof}

Nearly the same proof applies in the case of alternating link types, but we remind the reader that the enumerative and structural results of Sundberg and Thistlethwaite~\cite{Sundberg_1998} and Thistlethwaite~\cite{Thistlethwaite_1998} are more descriptive than the comparatively vague pattern theorem.

\subsection{The rarity of alternating knot types}
\label{sec:rarity}

The pattern theorem says much about the asymptotic structure of random prime alternating knot types. As our motivating application, we prove that alternating knots are rare:
\begin{reptheorem}{thm:altknotrare}
  All but exponentially few prime knot types are nonalternating. This furthermore implies that all but exponentially few knot types (prime or composite) are nonalternating.
\end{reptheorem}

\begin{proof}
  Consider the pattern tangle \(R\) and the nonalternating tangle \(\overline R\) in Figure~\ref{fig:altknot_altpattern}. The pattern theorem guarantees that there is \(c > 0\) so that almost every alternating knot tangle type \([T]\) has a tangle diagram \(T\) with at least \(cn\) instances of \(R\). By replacing any number of instances of \(R\) in an alternating diagram \(T\) by the tangle \(\overline R\), one obtains a non-alternating knot tangle diagram in a superclass of knot tangle diagrams \(\Dknot\) that still satisfies the Tait flyping conjecture. As a tangle diagram has no symmetries any diagram so produced corresponds to a \emph{distinct} knot type among all such diagrams produced from \(T\). Furthermore, any such diagrams produced from a different alternating knot tangle type \([T']\) will be distinct, by the solution to the flyping conjecture.

  \begin{figure}
    \centering
    \begin{subfigure}{.4\linewidth}
      \centering
      \includestandalone{figs/altknot_rtangle}
      \subcaption{The type-I tangle \(R\), which occurs linearly often in alternating knot types by the pattern theorem.}
    \end{subfigure}\hfil
    \begin{subfigure}{.4\linewidth}
      \centering
      \includestandalone{figs/altknot_rbartangle}
      \subcaption{The tangle \(\overline R\), equally likely as \(R\), can be found in the superclass \(\Dknot\) that still satisfies the flype conjecture.}
    \end{subfigure}
    \caption{The tangles \(R\) and \(\overline R\).}
    \label{fig:altknot_altpattern}
  \end{figure}

  For all but exponentially few knot tangle types \([T]\), we thus can produce at least \(2^{cn}\) unique nonalternating knot tangle types. Hence alternating knot tangles are rare among all knot tangles. Application of the same argument of Sundberg and Thistlethwaite~\cite{Sundberg_1998}, Section~3 shows then that alternating knot types are also rare among all knot types.
\end{proof}

\subsection{Subknots in alternating knot types}
\label{sec:subknots}

To demonstrate further the type of structural knowledge the pattern theorem provides, we provide new insight on questions of Millett and Jablan~\cite{Millett_2016,Millett_2017}. We note that the original questions are posed in the context of all minimal prime knot diagrams, for which an equivalent theorem follows by instead applying Theorem~\ref{thm:akdpatterntheorem}. The new result for alternating prime knot types is more powerful:
\begin{theorem}
  Let \([K]\) be an alternating knot type. Then all but exponentially few prime alternating knot types \([K]\) have the property that all minimal diagrams for \([L]\) contain \([K]\) as a subknot and \([K]\) as a slipknot.
\end{theorem}

\begin{proof}
  Given a construction of a type-I tangle \(S_K\) which introduces a slipknot of \([K]\), the pattern theorem yields the result. Figure~\ref{fig:altknot_doubleslip} describes this procedure. Let \(K\) be a minimal alternating diagram for \([K]\) and let \(K_o\) be an opening of \(K\) (that is, slice one edge into two loose legs).

  By changing each local neighborhood about a crossing into an alternating weaved ``multi-crossing'' as in Figure~\ref{fig:altknot_doubleslip_cross}, one obtains a 6-tangle whose legs counterclockwise are colored ABCCBA. By joining one pair of two adjacent C--A legs as in Figure~\ref{fig:altknot_doubleslip_cap} (introducing a new crossing with the B strand in compliance with the alternating structure) one obtains the alternating crossing tangle \(S_K\) admitting no flypes which can be seen to be type-I, as it is both nontrivial and at least 3-edge-connected.   Figure~\ref{fig:trefoil_subslip} shows as an example \(S_{3_1}\), which introduces trefoil subknots and slipknots. Application of Theorem~\ref{thm:aktpatterntheorem} together with \(S_K\) yields the result.
  \begin{figure}
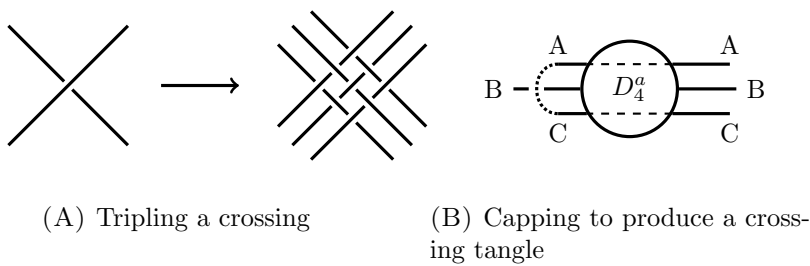

    \centering
    \begin{subfigure}[t]{0.4\textwidth}
      \centering
      \includestandalone{figs/tang_tripling}
      \subcaption{Tripling a crossing}
      \label{fig:altknot_doubleslip_cross}
    \end{subfigure}\hfil
    \begin{subfigure}[t]{0.4\textwidth}
      \centering
      \raisebox{.15in}{
      \includestandalone{figs/tang_tripling_cap}}
    \subcaption{Capping to produce a crossing tangle}
    \label{fig:altknot_doubleslip_cap}
    \end{subfigure}

    \caption{Rules for transforming a reduced open alternating knot diagram for \([K]\) into a tangle \(S_K\) that introduces \([K]\) slipknots and subknots into a larger diagram.}
    \label{fig:altknot_doubleslip}
  \end{figure}
 
  \begin{figure}
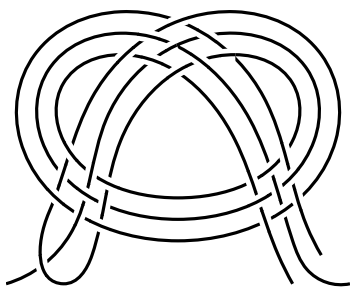

    \centering
    \includestandalone{figs/altknot_trefslip}
    \caption{A tangle which introduces a trefoil slipknot segment into an alternating prime knot diagram (and hence also a trefoil subknot segment).}
    \label{fig:trefoil_subslip}
  \end{figure}
\end{proof}

\section{Conclusion}
\label{sec:conclusion}

By restricting our view to only alternating prime knot types, we are able to prove an asymptotic structure theorem (the pattern theorem) for an infinite family of knot types (alternating knot types), rather than simply diagrams. We have then applied this pattern theorem to show that alternating knot types are rare.

The proofs for this subclass of knot types relies heavily on strong results about the structure of minimal diagrams under the Tait flyping conjecture. It is likely that similar results can be proved about other subclasses of knot types satisfying the Tait flyping conjecture; our restriction to alternating knot types only stems from that this subclass is the most well-known and well-studied such class. Extensions to other classes of knot types, such as satellite or hyperbolic knot types, requires additional work, as they do not satisfy the flyping conjecture. We hope that further understanding of knot diagrams can yield results in these cases, as there is evidence~\cite{malyutin2016question} that they are directly related to the long-standing question as to whether crossing number is additive under connected summation~\cite{Kirbyproblems}.

\section*{Acknowledgements}

The author is grateful to both Rob Kusner and Malik Obeidin for alerting him to an error in the statement of the main theorem.

\FloatBarrier

\begingroup
\raggedright{}
\sloppy
\printbibliography{}
\endgroup

\hrulefill

\end{document}
